\documentclass{amsart}
\usepackage{amsfonts}
\usepackage{latexsym}
\usepackage{amssymb}
\usepackage{amsmath}
\usepackage{color}

\newcommand{\R}{\mathbb R}

\newcommand{\N}{\mathbb N}

\newtheorem{thm}{Theorem}[section]
\newtheorem{lemma}[thm]{Lemma}
\newtheorem{proposition}[thm]{Proposition}

\newtheorem{df}{Definition}[section]
\theoremstyle{remark}
\newtheorem{rmk}{Remark}

\begin{document}

%%%%%%%%%%%%%%%%%%%%%%%%%%%%%%%%%%%%%%%%%%%%%5

\title{Zhang's inequality for log-concave functions}

\author[D. Alonso]{David Alonso-Guti\'{e}rrez}
\address{\'Area de An\'alisis Matem\'atico, Departamento de Matem\'aticas, Facultad de Ciencias, Universidad de Zaragoza, Pedro Cerbuna 12, 50009 Zaragoza (Spain), IUMA}
\email[(David Alonso-Guti\'errez)]{alonsod@unizar.es}
\thanks{The first and second authors are partially supported by MINECO Project MTM2016-77710-P, DGA E26\_17R and IUMA,
and the third author is partially supported by Fundaci\'on S\'eneca, Programme in Support of Excellence Groups of the Regi\'on de Murcia, Project 19901/GERM/15, and by MINECO Project MTM2015-63699-P}
\date{\today}

\author[J. Bernu\'es]{Julio Bernu\'es}
\address{\'Area de an\'alisis matem\'atico, Departamento de matem\'aticas, Facultad de Ciencias, Universidad de Zaragoza, Pedro cerbuna 12, 50009 Zaragoza (Spain), IUMA}
\email[(Julio Bernu\'es)]{bernues@unizar.es}

\author[B. Gonz\'alez]{Bernardo Gonz\'alez Merino}
\address{Departamento de An\'alisis Matem\'atico, Facultad de Matem\'aticas, Universidad de Sevilla, Apdo. 1160, 41080-Sevilla, Spain}
\email[Bernardo Gonz\'alez Merino]{bgonzalez4@us.es}
\begin{abstract}
Zhang's reverse affine isoperimetric inequality states that among all convex bodies $K\subseteq\R^n$, the affine invariant quantity $|K|^{n-1}|\Pi^*(K)|$ (where $\Pi^*(K)$ denotes the polar projection body of $K$) is minimized if and only if $K$ is a simplex. In this paper we prove an extension of Zhang's inequality in the setting of integrable log-concave functions, characterizing also the equality cases.
\end{abstract}
\maketitle
\section{Introduction}

Given a convex body (compact, convex with non-empty interior) $K\subseteq\R^n$, its \textit{polar projection body} $\Pi^*(K)$ is the unit ball of the norm given by
$$
\Vert x\Vert_{\Pi^*(K)}:=|x||P_{x^\perp}K|,\qquad x\in\R^n
$$
where $|\cdot|$ denotes both the \textit{Lebesgue measure} (in the suitable space) and the Euclidean norm, and $P_{x^\perp}K$ is the
\textit{orthogonal projection} of $K$  onto the hyperplane orthogonal to $x$. The \textit{Minkowski functional} of a convex body $K$ containing the origin is defined, for every $x\in\R^n$, as $\Vert x\Vert_K:=\inf\{\lambda>0\mid x\in\lambda K\}\in[0,\infty]$. It is a norm if and only if $K$ is centrally symmetric.

 The expression $|K|^{n-1}|\Pi^*(K)|$ is affine invariant and the extremal convex bodies are well known: \textit{Petty's projection inequality} \cite{P} states that the (affine class of the) $n$-dimensional Euclidean ball,  $B_2^n$, is the only maximizer and \textit{Zhang's inequality} \cite{Z1} (see also \cite{GZ} and \cite{AJV}) proves that the (affine class of the)  $n$-dimensional simplex $\Delta$ is the only minimizer. That is, for any convex body $K\subseteq\R^n$,
$$
\frac{\binom{2n}{n}}{n^n}=|\Delta|^{n-1}|\Pi^*(\Delta)|\leq|K|^{n-1}|\Pi^*(K)|\leq |B_2^n|^{n-1}|\Pi^*(B_2^n)|=\left(\frac{|B_2^n|}{|B_2^{n-1}|}\right)^n.
$$

In recent years, many relevant geometric inequalities have been extended to the more general context of \textit{log-concave functions}, i.e., functions $f:\R^n\to[0,\infty)$ of the form $f(x)=e^{-v(x)}$ with $v:\R^n\to (-\infty,\infty]$ a convex function. The set of all log-concave and integrable functions in $\R^n$ will be denoted by $\mathcal F(\R^n)$. This family of functions contains the set of convex bodies via the natural injections $K\to\chi_K$ (where $\chi_K$ denotes the characteristic function of $K$) or $K\to e^{-\Vert\cdot\Vert_K}$ (when $K$ is a convex body containing the origin). We refer the reader to \cite{KM} or \cite{C} and the references therein for a quick introduction on this topic.

  The aim of this paper is to provide an extension of Zhang's inequality for every $f\in\mathcal F(\mathbb R^n)$. For that matter, we  define (see \cite{AGJV}, \cite{KM}) for every $f\in\mathcal F(\mathbb R^n)$ the (centrally symmetric) \textit{polar projection body of $f$}, denoted  $\Pi^*(f)$, as
 $$ \Vert x\Vert_{\Pi^*(f)}=2|x|\int_{x^\perp}P_{x^\perp}f(y)dy,\quad x\in\R^n
  $$  where $P_{x^\perp}f:x^\perp\to[0,\infty)$ is \textit{the shadow of $f$}, i.e.,
   $
   P_{x^\perp}f(y):=\max_{s\in\R}f\left(y+s\frac{x}{|x|}\right)
   $.

Petty's projection inequality was extended by Zhang, see \cite{Z2}, to compact domains and it was shown to be equivalent to the so called
\textit{affine Sobolev inequality}, which for log-concave functions in the suitable Sobolev space takes the following form: For every $f\in\mathcal F(\mathbb R^n)\cap W^{1,1}(\R^n)=\left\{f\in\mathcal F(\mathbb R^n)\,:\frac{\partial f}{\partial x_i} \in L^1(\R^n) \,,\,\forall i\right\}$
$$
\Vert f\Vert_{\frac{n}{n-1}}|\Pi^*(f)|^\frac{1}{n}\leq\frac{|B_2^n|}{2|B_2^{n-1}|},
$$
with equality if and only if $f$ is the characteristic function of an ellipsoid.

Our extension of Zhang's inequality has the following form:

\begin{thm}\label{thm:FunctionalZhang} Let $f\in\mathcal F(\mathbb R^n)$. Then,
$$
\int_{\R^n}\int_{\R^n}\min\left\{f(y),f(x)\right\}dydx\leq 2^nn! \Vert f\Vert_\infty\Vert f\Vert_1^{n+1}\left|\Pi^*\left(f\right)\right|.
$$ Moreover, if $\Vert f\Vert_\infty=f(0)$ then equality holds if and only if $\frac{f(x)}{\Vert f\Vert_\infty}=e^{-\Vert x\Vert_{\Delta}}$ for some n-dimensional simplex $\Delta$ containing the origin.
\end{thm}

\begin{rmk}
The inequality in Theorem \ref{thm:FunctionalZhang} is affine invariant, i.e., the inequality does not change under compositions of $f$ with affine transformations of $\R^n$.
\end{rmk}

%\begin{rmk}
%\textcolor{red}{For every $f\in\mathcal F(\R^n)$ with $f(0)=\Vert f\Vert_\infty$ and such that
%$f\in \mathcal{C}^{1)}$ is defined in a compact support, then Theorem \ref{thm:FunctionalZhang} rewrites as
%\[
%\int_{\R^n}\int_{\R^n}\min\left\{\frac{f(y)}{\Vert f\Vert_\infty},\frac{f(x)}{\Vert f\Vert_\infty}\right\}dydx\leq \frac{2^nn!}{n^n} \left(\frac{\Vert f\Vert_1}{\Vert f\Vert_\infty}\right)^{n+1}\left(\int_{S^{n-1}} \Vert \nabla_u %f\Vert_1^{-n}du\right)^n.
%\]}
%\end{rmk}

\begin{rmk}\label{rmk:RecoversZhang} Theorem \ref{thm:FunctionalZhang} extends Zhang's inequality. Indeed, if
	$f(x)=e^{-\Vert x\Vert_K}$ for some convex body $K$ containing the origin,
then
$$\Pi^*(f)=\frac{1}{2(n-1)!}\Pi^*(K),\hspace{1cm}\Vert f\Vert_1=n!|K|,
$$
	and
\begin{align*}&\int_{\R^{n}}\int_{\R^n}\min\left\{f(y),f(x)\right\}dydx=\int_{\R^{2n}}e^{-\max\{\Vert x\Vert_K,\Vert y\Vert_K\}}dydx\cr
	&\quad\quad\quad=\int_{\R^{2n}}e^{-\Vert (x,y)\Vert_{K\times K}}dydx=(2n)!|K|^2.
\end{align*}
Thus, Theorem \ref{thm:FunctionalZhang} yields
	$$
	\frac{\binom{2n}{n}}{n^n}\leq |K|^{n-1}|\Pi^*(K)|.
	$$
\end{rmk}

\begin{rmk}Sharp
lower and upper bounds of the left hand side of the inequality in Theorem \ref{thm:FunctionalZhang} in terms of the integral
of $f$  are known (see Lemmas 2.3 and 2.9 in \cite{AAGJV}, respectively).
 \end{rmk}

The paper is structured as follows. In Section \ref{sec:NotationPreli} we introduce some notation and preliminary results. A crucial role in the proof of Theorem \ref{thm:FunctionalZhang} will be played by the functional form of the covariogram function, that we shall denote $g$, associated to any  $f\in\mathcal{F}(\R^n)$. Recall that in the geometric setting the covariogram function of a convex body $K$ is given by $\R^n\ni x\to |K\cap (x+K)|$. In this section we shall define and study the basic properties of its functional version $g$.

 In Section \ref{sec:IneqFuncZhang} we prove the inequality in Theorem \ref{thm:FunctionalZhang}. The proof will rely on the following two facts: First, we will show that the polar projection body of $f\in\mathcal{F}(\R^n)$ can be expressed in terms of dilations of the level sets of $g$. This can be seen as an extension of the corresponding geometric result (see \cite[Theorem 1]{SCH} and \cite[Propositions 4.1 and 4.3]{AJV}) where the polar projection body of a convex body appears as the limit of suitable dilations of the level sets of the covariogram function. Second, we will prove a sharp relation (by inclusion) between the level sets of $g$ and a convex body in the celebrated family of bodies introduced by Ball (cf.~\cite[Pg.~74]{B}). The proof of such inclusion, that we state in full generality, follows ideas from \cite[Lemmas 2.1 and 2.2]{KM}.

 In Section \ref{sec:EqFunctionalZhang} we characterize the equality cases. We first show that the equality in Theorem \ref{thm:FunctionalZhang} holds if and only if the function $g$ is log-linear on every 1-dimensional linear ray starting at 0, i.e., for every $x\in\R^n$ and every $\lambda\in[0,1]$, $g(\lambda x)=g(0)^{1-\lambda} g(x)^\lambda$ and then prove that such condition implies that $f$ has to be as in the statement of the theorem.

\section{Notation and preliminaries}\label{sec:NotationPreli}

$S^{n-1}$ is the Euclidean unit sphere in $\R^n$ and $\sigma$ denotes the uniform probability measure on $S^{n-1}$.
If the origin is in the interior of $K$, the function $\rho_K\colon S^{n-1}\to [0,+\infty)$ given by $\rho_K(u)=\sup\{\lambda\ge 0\mid \lambda u\in K\}$ is the radial function of $K$. It extends to $\R^{n}\setminus\{0\}$ via $t\rho_K(tu)=\rho_K(u)$, for any $t>0,u\in S^{n-1}$. The volume of $K$ is given by $$|K|=|B_2^n|\int_{S^{n-1}}\rho^n_K(u) d\sigma(u)$$ and the boundary of $K$ will be denoted by $\partial K$.

Throughout the paper, $f:\R^n\to[0,+\infty)$ will always denote a (non identically null) log-concave integrable function.
 Recall that such $f$ is said to be log-concave if it can be written as $\frac{f}{\Vert f\Vert_\infty}=e^{-v}$ where $v:\R^n\to(-\infty,+\infty]$ is convex or, equivalently if for every $x,y\in\R^n, 0<\lambda<1$, $f(\lambda x+(1-\lambda)y)\ge (f(x))^{\lambda}(f(y))^{1-\lambda}$.  It is well known that $f$ is then continuous in the interior of its support, $\textrm{int}(\textrm{supp} f)$, and we denote its supremum as $\Vert f\Vert_{\infty}$.
Since Theorem \ref{thm:FunctionalZhang} does not depend on the values of $f$ on the boundary of $\textrm{supp} f$ we shall assume, without loss of generality, that $f$ is continuous on its support.
 For any $t\in[0,\infty)$ we denote the level sets of $v$ by
 $$
 K_t(f)=\{x\in\R^n\,:\, f(x)\geq e^{-t}\Vert f\Vert_\infty\}=\{x\in\R^n\,:\, v(x)\le t\}.
 $$

 Since we assume that $f$ is continuous on its support, $K_t(f)$ is a convex body for all $t>0$. These and other basic facts on convex bodies and log-concave functions used in the paper can be found in \cite{BGVV}.

 We will use the following definition of the polar projection body of $f$ which involves level sets, equivalent to the one stated in the introduction (see \cite[Proposition 4.1]{AGJV}).

   \begin{df}\label{projection}
  Let $f\in\mathcal F(\mathbb R^n)$. The polar projection body of $f$, denoted as $\Pi^*(f)$, is the  unit ball of the norm given by
  $$\Vert x\Vert_{\Pi^*(f)}:=2|x|\Vert f\Vert_{\infty}\!\int_{0}^{\infty}\!|P_{x^\perp}K_t(f)|\ e^{-t}dt=
 2\Vert f\Vert_{\infty}\int_{0}^{\infty}\!\Vert x\Vert_{\Pi^*(K_t(f))}e^{-t}dt. $$
   \end{df}

 %\begin{rmk} For every $f$ in the suitable Sobolev space, an alternative expression is $\displaystyle \Vert x\Vert_{\Pi^*(f)}=\int_{\R^n}|\langle\nabla f(y),x\rangle| dy,$ (see also \cite[Proposition 4.1]{AGJV}).
 %\end{rmk}

 \begin{rmk}
 	 If $f=\chi_K$ is the characteristic function of a convex body $K$, then $\Pi^*(f)=\frac{1}{2}\Pi^*(K)$ and, as mentioned in Remark \ref{rmk:RecoversZhang}, if $f(x)=e^{-\Vert x\Vert_K}$, then $\Pi^*(f)=\frac{1}{2(n-1)!}\Pi^*(K)$.
 \end{rmk}

We start by associating a function $g$ to any function $f\in\mathcal{F}(\R^n)$. Such function can be regarded as the functional version of the covariogram functional. We collect its properties in the following lemma
\begin{lemma}\label{lem:g(x)}
Let $f\in\mathcal F(\mathbb R^n)$. Then the function $g:\R^n\to\R$ defined by
$$
g(x):=\int_0^\infty e^{-t}|K_t(f)\cap(x+ K_t(f))|dt=\int_{\R^n}\min\left\{\frac{f(y)}{\Vert f\Vert_\infty},\frac{f(y-x)}{\Vert f\Vert_\infty}\right\}dy
$$
is even, log-concave, $0\in\textrm{int}(\textrm{supp}\ g)$ with $\displaystyle \Vert g\Vert_\infty=g(0)=\int_0^\infty e^{-t}|K_t(f)|dt=\int_{\R^n}\frac{f(x)}{\Vert f\Vert_\infty}dx>0$, and $\displaystyle{\int_{\R^n}g(x)dx=\int_{\R^n}\int_{\R^n}\min\left\{\frac{f(y)}{\Vert f\Vert_\infty},\frac{f(x)}{\Vert f\Vert_\infty}\right\}dydx}$.
\end{lemma}

\begin{proof}
	
	By Fubini's theorem, for any $x\in\R^n$ we have
	\begin{equation*}
     \begin{split}
		\int_0^\infty e^{-t}|K_t(f)\cap(x+ K_t(f))|dt & =\int_{\R^n}\int_{-\log\min\left\{\frac{f(y)}{\Vert f\Vert_\infty},\frac{f(y-x)}{\Vert f\Vert_\infty}\right\}}^\infty e^{-t}dtdx \\
		& = \int_{\R^n}\min\left\{\frac{f(y)}{\Vert f\Vert_\infty},\frac{f(y-x)}{\Vert f\Vert_\infty}\right\}dy.
	  \end{split}
    \end{equation*}
	Consequently, using Fubini's theorem and a change of variables,
$$
\int_{\R^n}g(x)dx=\int_{\R^n}\int_{\R^n}\min\left\{\frac{f(y)}{\Vert f\Vert_\infty},\frac{f(x)}{\Vert f\Vert_\infty}\right\}dydx.
$$
	In order to prove the log-concavity of $g$, let $x_1,x_2\in\R^n$, $t_1,t_2\in[0,\infty)$, $0\leq\lambda\leq1$ and write $x=(1-\lambda)x_1+\lambda x_2$ and $t=(1-\lambda)t_1+\lambda t_2$. Clearly,
	$$
	K_t\cap(x+ K_t)\supseteq(1-\lambda)(K_{t_1}\cap(x_1+K_{t_1}))+\lambda (K_{t_2}\cap(x_2+K_{t_2})).
	$$
By Brunn-Minkowski inequality,  \cite[Theorem 1.2.1]{BGVV},
	$$
	|K_t\cap(x+ K_t)|\geq|K_{t_1}\cap(x_1+K_{t_1})|^{1-\lambda}|K_{t_2}\cap(x_2+K_{t_2})|^{\lambda}.
	$$
	Thus,
	$$
	e^{-t}|K_t\cap(x+ K_t)|\geq (e^{-t_1}|K_{t_1}\cap(x_1+K_{t_1})|)^{1-\lambda}(e^{-t_2}|K_{t_2}\cap(x_2+K_{t_2})|)^{\lambda}
	$$
	and, by Pr\'ekopa-Leindler inequality, \cite[Theorem 1.2.3]{BGVV},
	\begin{align*}
		&g((1-\lambda)x_1+\lambda x_2)=\int_0^\infty e^{-t}|K_t\cap((1-\lambda)x_1+\lambda x_2+ K_t)|dt\ge\cr
		&\quad\quad\quad\geq\left(\int_0^\infty e^{-t}|K_t\cap(x_1+ K_t)|dt\right)^{1-\lambda}\left(\int_0^\infty e^{-t}|K_t\cap(x_2+ K_t)|dt\right)^{\lambda}\cr
		&\quad\quad\quad=g(x_1)^{1-\lambda}g(x_2)^\lambda.
	\end{align*}

Now, for any $t\in[0,\infty)$, $K_t(f)\cap(x+K_t(f))=x+ (K_t(f)\cap(-x+K_t(f)))$ and so $|K_t(f)\cap(x+K_t(f))|=|K_t(f)\cap(-x+K_t(f))|$. Therefore
$g(x)=g(-x)$.
Consequently, $\Vert g\Vert_\infty=g(0)$ and its value is
$$
g(0)=\int_0^\infty e^{-t}|K_t(f)|dt=\int_{\R^n}\frac{f(x)}{\Vert f\Vert_\infty}dx>0.
$$
Finally, there exists $\varepsilon>0$ such that if $|x|<\varepsilon$ then $K_1(f)\cap(x+ K_1(f))$ is a non-empty convex body and has positive volume. Thus, if $|x|<\varepsilon$ then for every $t>1$ $K_t(f)\cap(x+ K_t(f))$ has positive volume and then $g(x)>0$. Thus $0\in\textrm{int}(\textrm{supp}g)$.
\end{proof}

Let $g\in\mathcal F(\mathbb R^n)$ be a function
such that $g(0)>0$ and let $p>0$.
The following important family of convex bodies was introduced  by K. Ball in \cite[pg.~74]{B}. We denote
$$
\widetilde{K}_p(g):=\left\{x\in\R^n\,:\, \int_0^\infty g(rx)r^{p-1}dr\geq\frac{g(0)}{p}\right\}.
$$
It follows from the definition that the radial function of $\widetilde{K}_p(g)$ is given by
$$
\rho^p_{\widetilde{K}_p(g)}(u)=\frac{1}{g(0)}\int_0^\infty pr^{p-1}g(rx)dr.
$$
 We will make use of the following well-known relation between the Lebesgue measure of $\widetilde{K}_n(g)$ and the integral of $g$.

\begin{lemma}[\cite{B}]\label{VolumeAndIntegral}
	Let $g\in\mathcal F(\mathbb R^n)$ be such that $g(0)>0$. Then
	$$
	|\widetilde{K}_n(g)|=\frac{1}{g(0)}\int_{\R^n}g(x)dx.
	$$
\end{lemma}
\begin{proof}
	Integrating in polar coordinates we have that
	\begin{eqnarray*}
	|\widetilde{K}_n(g)|&=&|B_2^n|\int_{S^{n-1}}\rho^n_{\widetilde{K}_n(g)}(u)d\sigma(u)=|B_2^n|\int_{S^{n-1}}\frac{n}{g(0)}\int_0^\infty r^{n-1}g(rx)drd\sigma(u)\cr
	&=&\frac{1}{g(0)}\int_{\R^n}g(x)dx.
\end{eqnarray*}
\end{proof}

\section{Proof of the inequality in Theorem \ref{thm:FunctionalZhang}}\label{sec:IneqFuncZhang}

 We split the main idea of the proof in two parts. We first prove, Lemma \ref{lem:K_g}, that $\Pi^*(f)$ equals the intersection of suitable dilations of the level sets ${K}_t(g)$, with $g$ as defined in the previous section. We then show a sharp relation between Ball's convex body $\widetilde{K}_n(g)$ and the level set $K_t(g)$, see Lemma \ref{Inclusion}. Such a relation holds not only for the covariogram function $g$ but for a larger class of log-concave functions.

\begin{lemma}\label{lem:K_g}
Let $f\in\mathcal F(\mathbb R^n)$ and let $g:\R^n\to\R$ be the function
$$
g(x)=\int_0^\infty e^{-t}|K_t(f)\cap(x+ K_t(f))|dt.
$$
Then for every $0<\lambda_0<1$
$$
\bigcap_{0<\lambda<\lambda_0}\frac{K_{-\log(1-\lambda)}(g)}{\lambda}=2\Vert f\Vert_1\,\,\Pi^*(f).
$$
\end{lemma}
\begin{proof}
For any $0<\lambda<1$ the convex body $\frac{K_{-\log(1-\lambda)}(g)}{\lambda}$ equals
$$\left\{x\in\R^n\,:\,\int_0^\infty e^{-t}|K_t(f)\cap(\lambda x+ K_t(f))|dt\geq(1-\lambda)\int_0^\infty e^{-t}|K_t(f)|dt\right\}$$
or equivalently
$$\left\{x\in\R^n\,:\,\int_0^\infty\!  e^{-t}\frac{|K_t(f)|-|K_t(f)\cap(\lambda |x| \frac{x}{|x|}+ K_t(f))|}{\lambda}dt\leq\int_0^\infty\! e^{-t}|K_t(f)|dt\right\}.
$$

Since $|K_t(f)|-|K_t(f)\cap(\lambda |x| \frac{x}{|x|}+ K_t(f))|\leq \lambda |x| |P_{x^\perp}K_t(f)|$, then
\begin{eqnarray*}
\int_0^\infty e^{-t}\frac{|K_t(f)|-|K_t(f)\cap(\lambda |x| \frac{x}{|x|}+ K_t(f))|}{\lambda}dt&\leq&|x|\int_0^\infty e^{-t}|P_{x^\perp}K_t(f)|dt\cr &=&\frac{\Vert x\Vert_{\Pi^*(f)}}{2\Vert f\Vert_\infty}
\end{eqnarray*}
and we have that if $$\Vert x\Vert_{\Pi^*(f)}\leq 2\Vert f\Vert_\infty\int_0^\infty e^{-t}|K_t(f)|dt=2\int_{\R^n}f(x)dx$$
then $x\in K_{-\log(1-\lambda)}(g)/\lambda$. Thus
$$2\Vert f\Vert_1\Pi^*(f)\subseteq\frac{K_{-\log(1-\lambda)}(g)}{\lambda}, \,\,\text{ for all }0<\lambda<1.$$

On the other hand, for any $0<\lambda<1$ and any $x\in\R^n,$
\begin{eqnarray*}
&&\int_0^\infty e^{-t}\frac{|K_t(f)|-|K_t(f)\cap(\lambda|x|\frac{x}{|x|}+K_t(f))|}{\lambda}dt\cr
&\geq&|x|\int_0^\infty e^{-t}|P_{x^\perp}\big(K_t(f)\cap(\lambda x+ K_t(f))\big)|dt
\end{eqnarray*}
and then, since  $\int_0^\infty e^{-t}|P_{x^\perp}\big(K_t(f)\cap(\lambda x+ K_t(f))\big)|dt$  decreases in $\lambda$ and
\begin{eqnarray*}
\sup_{\lambda\in(0,1)}|x|\int_0^\infty e^{-t}|P_{x^\perp}\big(K_t(f)\cap(\lambda x+ K_t(f))\big)|dt&=&|x|\int_0^\infty e^{-t}|P_{x^\perp}K_t(f)|dt\cr
&=&\frac{\Vert x\Vert_{\Pi^*(f)}}{2\Vert f\Vert_\infty},
\end{eqnarray*}
%\begin{eqnarray*}
%&&\sup_{\lambda\in(0,1)}\int_0^\infty e^{-t}\frac{|K_t(f)|-|K_t(f)\cap(\lambda|x|\frac{x}{|x|}+K_t(f))|}{\lambda}dt\cr
%&\geq&\sup_{\lambda\in(0,1)}|x|\int_0^\infty e^{-t}|P_{x^\perp}\big(K_t(f)\cap(\lambda x+ K_t(f))\big)|dt\cr
%&=&|x|\int_0^\infty e^{-t}|P_{x^\perp}K_t(f)|dt=\frac{\Vert x\Vert_{\Pi^*(f)}}{2\Vert f\Vert_\infty}.
%\end{eqnarray*}
we have that if $\displaystyle\Vert x\Vert_{\Pi^*(f)}>2\Vert f\Vert_\infty\int_0^\infty e^{-t}|K_t(f)|dt=2\Vert f\Vert_1$, there exists $\lambda_1>0$ such that for every $0<\lambda\leq\lambda_1$
$$
\int_0^\infty e^{-t}\frac{|K_t(f)|-|K_t(f)\cap(\lambda|x|\frac{x}{|x|}+K_t(f))|}{\lambda}dt>\int_0^\infty e^{-t}|K_t(f)|dt
$$
and then $x\not\in \frac{K_{-\log(1-\lambda)}(g)}{\lambda}$ if $0<\lambda\leq\lambda_1$.
\end{proof}

In the following lemma we prove the aforementioned inclusion between the level sets of a function $g\in\mathcal{F}(\R^n)$ and the convex body $\widetilde{K}_n(g)$. We follow ideas from \cite{KM}. In that paper a similar result was stated with an interest on large values of $t$. Here we shall be interested on small values of $t$. In the second part of the Lemma we provide information on the equality case as it shall be used in the next Section.

\begin{lemma}\label{Inclusion}
Let $g\in\mathcal F(\mathbb R^n)$ be such that $0\in\textrm{int}(\textrm{supp}g)$, $g(0)=\Vert g\Vert_{\infty}>0$. Then for every $0\leq t\leq\frac{n}{e}$,
$$
\frac{t}{(n!)^\frac{1}{n}}\widetilde{K}_n(g)\subseteq K_{t}(g).
$$
Moreover, for any $0<t\leq\frac{n}{e}$ there is equality if and only if $g$ is log-linear on every 1-dimensional linear ray starting at 0 and, furthermore, if $g$ is not log-linear on the 1-dimensional linear ray starting at 0 spanned by $u\in S^{n-1}$, there exists $\varepsilon>0$ such that for every $0<t\leq \frac{n}{e}$
$$
\frac{t}{(n!)^\frac{1}{n}}(\rho_{\widetilde{K}_n(g)}(u)+\varepsilon)\leq\rho_{K_{t}(g)}(u).
$$
\end{lemma}

\begin{proof}
We can assume without loss of generality that $g(0)=1$. Otherwise consider $\frac{g}{\Vert g\Vert_\infty}$.  Write $g(x)=e^{-v(x)}$ for some convex function $v$ and fix $u\in S^{n-1}$.

 For any $q>0$ the function
$$
\phi(r)=v(ru)-q\log r
$$
is strictly convex in $(0,\infty)$ and $\omega(r):=v(ru)$ is non-decreasing and convex on $[0,\infty)$. Since $r^{q}e^{-v(ru)}$ is integrable on $[0,\infty)$ and takes value 0 at 0,  we have $\lim_{r\to\infty}\phi(r)=\lim_{r\to0^+}\phi(r)=\infty$ and consequently $\phi$ attains a unique minimum at some number $r_0=r_0(q)$ and the lateral derivatives of $\phi$ verify
$\phi'_{-}(r_0)\leq 0$ and  $\phi'_{+}(r_0)\geq 0$.
This implies that the lateral derivatives of $\omega(r)$ at $r_0$ verify
$\omega'_{-}(r_0)\leq \frac{q}{r_0}$,
and $\omega'_{+}(r_0)\geq \frac{q}{r_0}$. Notice that if $\omega$ is linear then necessarily $\omega(r)=\omega(r_0)+\frac{q}{r_0}(r-r_0)$.

 Notice also that $r\to \omega(r_0)+\frac{q}{r_0}(r-r_0)$ is a supporting line of $\omega$ at $r_0$ and by convexity, $\omega(r)\geq \omega(r_0)+\frac{q}{r_0}(r-r_0)$ for every $r\in [0,\infty)$. Thus,
\begin{align*}\label{InequalityOfIntegrals}
n\int_0^\infty r^{n-1}g(ru)dr&=n\int_0^\infty r^{n-1}e^{-\omega(r)}dr\leq ne^{q-\omega(r_0)}\int_0^\infty r^{n-1}e^{-\frac{qr}{r_0}}dr\cr
&=ne^{q-\omega(r_0)}\left(\frac{r_0}{q}\right)^n\int_0^\infty r^{n-1}e^{-r}dr=g(r_0u)\frac{e^{q}n!}{q^n}r_0^n.
\end{align*}
Moreover, the previous inequality is equality if and only if $\omega(r)= \omega(r_0)+\frac{q}{r_0}(r-r_0)$ for every $r\in[0,\infty)$.

Consequently, for any $q>0$ and since $
\rho^n_{\widetilde{K}_n(g)}(u)=n\int_0^\infty r^{n-1}g(ru)dr,
$
$$
\frac{q}{e^{\frac{q}{n}}(n!)^\frac{1}{n}}\rho_{\widetilde{K}_n(g)}(u)\leq g(r_0u)^\frac{1}{n}r_0,
$$
with equality if and only if $\omega(r)= \omega(r_0)+\frac{q}{r_0}(r-r_0)$ for every $r\in[0,\infty)$.
On the other hand, since $\omega$ is convex, $\omega'_{+}(r)\leq\frac{q}{r_0}$  if $r<r_0$  and then
$$
\omega(r_0)=\omega(0)+\int_0^{r_0}\omega'_{+}(r)dr\leq v(0)+q=q,
$$
 thus $g(r_0u)\geq e^{-q}$, with equality if and only if $\omega'_{+}(r)=\frac{q}{r_0}$ for every $r\in[0,r_0)$ and therefore $\omega(r)= \omega(r_0)+\frac{q}{r_0}(r-r_0)$ for every $r\in[0,r_0]$. Now, the definition of the log-concavity of $g$ applied to 0 and $r_0u$ yields
$$
g\left( g(r_0u)^\frac{1}{n}r_0u\right)\geq g(r_0u)^{g(r_0u)^\frac{1}{n}}\cdot 1=e^{g(r_0u)^\frac{1}{n}\log g(r_0u)}
$$
with equality if and only if $\omega(r)=v(ru)$ is linear in $[0,r_0]$ and then  $\omega(r)= \omega(r_0)+\frac{q}{r_0}(r-r_0)$ for every $r\in[0,r_0]$.
Since the function $x\to x^\frac{1}{n}\log x$ attains its minimum at $x=e^{-n}$ and is increasing in the interval $(e^{-n},\infty)$, if $0<q\leq n$
$$
g\left( g(r_0u)^\frac{1}{n}r_0u\right)\geq e^{-qe^{-\frac{q}{n}}},
$$
with equality if and only if $\omega(r)$ is linear in $[0,r_0]$ and $g(r_0u)=e^{-q}$, which occurs if and only if $\omega(r)= \omega(r_0)+\frac{q}{r_0}(r-r_0)$ for every $r\in[0,r_0]$, that is, $g(r_0u)^\frac{1}{n}r_0u\in K_{qe^{-\frac{q}{n}}}(g)$ and $g(r_0u)^\frac{1}{n}r_0u\in \partial K_{qe^{-\frac{q}{n}}}(g)$ if and only if  $\omega(r)= \omega(r_0)+\frac{q}{r_0}(r-r_0)$ for every $r\in[0,r_0]$.

Since this is valid for any $u\in S^{n-1}$,
$$
\frac{qe^{-\frac{q}{n}}}{(n!)^\frac{1}{n}}\widetilde{K}_n(g)\subseteq K_{qe^{-\frac{q}{n}}}(g),
$$
with equality if and only if for every $u\in S^{n-1}$,  $\omega(r)= \omega(r_0)+\frac{q}{r_0}(r-r_0)$ for every $r\in[0,\infty)$, i.e. $v(ru)$ is linear for every $u\in  S^{n-1}$.
 Finally, observe that when $x\in(0,n]$, the function $x\to xe^{-\frac{x}{n}}$ takes every value in $(0,\frac{n}{e}]$ thus, for every $0\leq t\leq \frac{n}{e}$
$$
\frac{t}{(n!)^\frac{1}{n}}\widetilde{K}_n(g)\subseteq K_{t}(g)
$$
and for any $t\in(0,\frac{n}{e}]$ there is equality if and only if $v(ru)$ is linear for every $u\in  S^{n-1}$, i.e., for every $x\in\R^n$ and every $\lambda\in[0,1]$, $g(\lambda x)=g(0)^{1-\lambda} g(x)^\lambda$.

In order to establish the \textit{furthermore} part, we are given some $u\in S^{n-1}$. We first need to prove the following

\textbf{Claim:} \textsl{The function $q\to r_0(q)$ is continuous in $(0,\infty)$ and is bounded around 0.}

Indeed, we first consider a sequence $(q_k)_{k=1}^\infty$  converging to $q\in(0,\infty)$ so that there is a subsequence $(r_0(q_{k_i}))_{i=1}^\infty$  converging to some $\overline{r}$. We have
$$
\omega(r_0(q_{k_i}))-q_{k_i}\log (r_0(q_{k_i}))\leq\omega(r_0(q))-q_{k_i}\log (r_0(q))
$$
and taking limits,
$
\omega(\overline{r})-q\log (\overline{r})\leq\omega(r_0(q))-q\log (r_0(q)).$ Therefore, $\overline{r}=r_0(q)$.

If, on the other hand, $q>0$ and the subsequence $(r_0(q_{k_i}))_{i=1}^\infty$ tends to $\infty$ then, since $q_k\leq M$ for every $k\in\N$ and some $M>0$, we have that
\begin{align*}
\omega(r_0(q_{k_i}))&-M\log (r_0(q_{k_i}))+(M-q_{k_i})\log (r_0(q_{k_i}))=\omega(r_0(q_{k_i}))-q_{k_i}\log (r_0(q_{k_i}))\cr
&\leq\omega(r_0(q))-q_{k_i}\log (r_0(q)),
\end{align*}
leading to a contradiction, since the left hand side of the inequality tends to $\infty$. Thus, both the inferior limit and the superior limit of $r_0(q_k)$ are equal to $r_0(q)$ and we have proven continuity in $(0,\infty)$. Finally, if $(q_k)_{k=1}^\infty$ is a sequence converging to 0 and some subsequence $(r_0(q_{k_i}))_{i=1}^\infty$ tends to $\infty$,  we would have that for every $r\in[0,\infty)$
$$
\omega(r_0(q_{k_i}))+\frac{q_{k_i}}{r_0(q_{k_i})}(r-r_0(q_{k_i}))\leq\omega(r),
$$
leading to a contradiction since the left hand side of this inequality tends to $\infty$. This finishes the proof of the Claim.

As a consequence, if $(q_k)_{k=1}^\infty$ converges to $q\in[0,\infty)$, we have that the sequence $(\frac{q_{k}}{r_0(q_{k})})_{k=1}^\infty$ is bounded, since
$
\frac{q_{k}}{r_0(q_{k})}\leq\omega_+^\prime(r_0(q_{k})).
$

Now, assume that there is no $\overline{\varepsilon}>0$ verifying that for every $0<q\leq n$
$$
n\int_0^\infty r^{n-1}e^{-(\omega(r_0)+\frac{q}{r_0}(r-r_0))}dr-n\int_0^\infty r^{n-1}e^{-\omega(r)}dr\geq\overline{\varepsilon}.
$$
Then we can find a sequence $(q_k)_{k=1}^\infty$ (and if necessary extract from it further subsequences which we denote in the same way) so that
$$
\lim_{k\to\infty}n\int_0^\infty r^{n-1}\left(e^{-(\omega(r_0(q_k))+\frac{q_k}{r_0(q_k)}(r-r_0(q_k)))}-e^{-\omega(r)}\right)dr=0,
$$
$q_k$ converges to some $q\in[0,n]$, $r_0(q_k)$ converges to some $\overline{r}\in[0,\infty)$ and $\frac{q_k}{r_0(q_k)}$ converges to some $\alpha\in[0,\infty)$. Therefore, since for every $r\in[0,\infty)$
$$
\omega(r_0(q_{k}))+\frac{q_{k}}{r_0(q_{k})}(r-r_0(q_{k}))\leq\omega(r),
$$
we have that for every $r\in[0,\infty)$, $
\omega(\overline{r})+\alpha(r-\overline{r})\leq\omega(r)
$
and, since by Fatou's lemma
\begin{eqnarray*}
	0&\leq& n\int_0^\infty r^{n-1}\left(e^{-(\omega(\overline{r})+\alpha(r-\overline{r}))}-e^{-\omega(r)}\right)dr\cr
	&\leq&\lim_{k\to\infty}n\int_0^\infty r^{n-1}\left(e^{-(\omega(r_0(q_k))+\frac{q_k}{r_0(q_k)}(r-r_0(q_k)))}-e^{-\omega(r)}\right)dr=0,
\end{eqnarray*}
we have that for every $r\in[0,\infty),
\omega(r)=\omega(\overline{r})+\alpha(r-\overline{r})
$
and so $\omega$ is linear.

Therefore if $\omega$ is not linear, there exists $\overline{\varepsilon}>0$ such that for every $0<q\leq n$,
$$
n\int_0^\infty r^{n-1}g(ru)dr+\overline{\varepsilon}\leq g(r_0u)\frac{e^{q}n!}{q^n}r_0^n
$$
 and so, for some $\varepsilon>0$ and every $0<q\leq n$
 $$
 \frac{q}{e^{\frac{q}{n}}(n!)^\frac{1}{n}}(\rho_{\widetilde{K}_n(g)}(u)+\varepsilon)\leq \frac{q}{e^{\frac{q}{n}}(n!)^\frac{1}{n}}(\rho^n_{\widetilde{K}_n(g)}(u)+\overline{\varepsilon})^\frac{1}{n}\leq g(r_0u)^\frac{1}{n}r_0.
 $$

Now, we continue as in the proof of the first part of the Lemma. If for some $u\in S^{n-1}$ we assume that $\omega$ is not linear, then there exists $\varepsilon>0$ such that for every $0<q\leq n$
$$
\frac{qe^{-\frac{q}{n}}}{(n!)^\frac{1}{n}}(\rho_{\widetilde{K}_n(g)}(u)+\varepsilon)\leq\rho_{K_{qe^{-\frac{q}{n}}}(g)}(u).
$$
and consequently,  for every $0<t\leq \frac{n}{e}$
$$
\frac{t}{(n!)^\frac{1}{n}}(\rho_{\widetilde{K}_n(g)}(u)+\varepsilon)\leq\rho_{K_{t}(g)}(u)
$$

\end{proof}

\begin{rmk}
The inclusion in the lemma above cannot be extended in general to the whole range of $t\in[0,n]$. If $f=\chi_K$ is the characteristic function of a convex body, $\widetilde{K}_n(f)=K$ and we would have, taking $t=n$,
$
\frac{n}{(n!)^\frac{1}{n}}\widetilde{K}_n(f)\subseteq K_{n}(f)
$
then, by using Stirling's formula, it would imply for large values of $n$ that
$
\frac{e}{2}K\subseteq K
$
which is trivially false.

On the other hand, using the same ideas as above we can obtain a more general result, namely, for every $p>0$ and $0\leq t\leq\frac{p}{e}$,
	$
	\frac{t}{\Gamma(1+p)^\frac{1}{p}}\widetilde{K}_p(g)\subseteq K_{t}(g).
	$

\end{rmk}

Now we can prove the main result of the paper. In this section we prove the inequality.

\begin{proof}[Proof of the inequality of Theorem \ref{thm:FunctionalZhang}]
Let us consider the function $g:\R^n\to\R$ defined by
\[
g(x)=\int_0^\infty e^{-t}|K_t(f)\cap(x+ K_t(f))|dt.
\]
By Lemma \ref{lem:K_g} and Lemma \ref{Inclusion} (since $g(0)=\Vert g\Vert_{\infty}>0$ by Lemma \ref{lem:g(x)}), for any $0<\lambda_0<1-e^{-\frac{n}{e}}$ we have
\begin{equation*}
2\Vert f\Vert_1\Pi^*\left(f\right)=\!\bigcap_{0<\lambda<\lambda_0}\frac{K_{-\log(1-\lambda)}(g)}{\lambda}
\supseteq\bigcap_{0<\lambda<\lambda_0}\frac{-\log(1-\lambda)}{(n!)^{\frac1n}\lambda}\widetilde{K}_n(g).
\end{equation*}
Since $h(\lambda):=-(\log(1-\lambda))/\lambda$ is increasing in $\lambda\in(0,1)$, and $\lim_{\lambda\rightarrow 0+}h(\lambda)=1$,
then
\[
2\Vert f\Vert_1\Pi^*\left(f\right)\supseteq \frac{1}{(n!)^{\frac1n}}\widetilde{K}_n(g).
\]
Taking Lebesgue measure we obtain that
$$
2^n\Vert f\Vert_1^n\left|\Pi^*\left(f\right)\right|\geq\frac{1}{n!}|\widetilde{K}_n(g)|.$$

One can conclude the result as a direct consequence of Lemmas \ref{lem:g(x)} and \ref{VolumeAndIntegral}.
\end{proof}

\section{Characterization of the equality in Theorem \ref{thm:FunctionalZhang}}\label{sec:EqFunctionalZhang}

In this section we characterize the equality case in Theorem \ref{thm:FunctionalZhang}. First we will show that if there is equality in the theorem for a function $f$ attaining its maximum at the origin, then the associated covariogram function $g$ has to be log-linear in every 1-dimensional ray starting at 0. Second we will prove that such condition implies that $\frac{f}{\Vert f\Vert_\infty}=e^{-\Vert\cdot\Vert_\Delta}$ for some simplex $\Delta$ containing the origin.

\begin{lemma}\label{lem:g(tx)=g(x)^t}
Let $f\in\mathcal F(\R^n)$ be such that $\Vert f\Vert_\infty=f(0)$ and let $g:\R^n\to\R$ be
\[
g(x)=\int_0^\infty e^{-t}|K_t(f)\cap(x+ K_t(f))|dt.
\]
If $f$ attains equality in Theorem \ref{thm:FunctionalZhang},
then for every $x\in\R^n$ and every $\lambda\in[0,1]$, $g(\lambda x)=g(0)^{1-\lambda} g(x)^\lambda$.
\end{lemma}

\begin{proof}
Assume that the statement is not true. Then, by Lemma \ref{Inclusion}, there exists $u\in S^{n-1}$ and $\varepsilon>0$ such that for any $0<\lambda_0<1-e^{-\frac{n}{e}}$ and any $0<\lambda<\lambda_0$
$$
\frac{-\log(1-\lambda)}{\lambda(n!)^\frac{1}{n}}(\rho_{\widetilde{K}_n(g)}(u)+\varepsilon)\leq\rho_{\frac{K_{-\log(1-\lambda)}(g)}{\lambda}}(u).
$$
Consequently, for such $u$
\begin{eqnarray*}
\rho_{\bigcap_{0<\lambda<\lambda_0}\frac{K_{-\log(1-\lambda)}(g)}{\lambda}}(u)&=&\inf_{0<\lambda<\lambda_0}\rho_{\frac{K_{-\log(1-\lambda)}(g)}{\lambda}}(u)\cr
&\geq&\inf_{0<\lambda<\lambda_0}\frac{-\log(1-\lambda)}{\lambda(n!)^\frac{1}{n}}(\rho_{\widetilde{K}_n(g)}(u)+\varepsilon)\cr
&=&\frac{1}{(n!)^\frac{1}{n}}(\rho_{\widetilde{K}_n(g)}(u)+\varepsilon),
\end{eqnarray*}
and then
$$
2\Vert f\Vert_1\Pi^*\left(f\right)=\bigcap_{0<\lambda<\lambda_0}\frac{K_{-\log(1-\lambda)}(g)}{\lambda}
\supsetneq\frac{1}{(n!)^{\frac1n}}\widetilde{K}_n(g)
$$
and the volume of the left-hand side convex body is strictly greater than the volume of the right hand side convex body.
\end{proof}

The next lemma shows that if $g$ is log-linear in every 1-dimensional linear ray starting from 0 then $\frac{f}{\Vert f\Vert_\infty}=e^{-\Vert\cdot\Vert_\Delta}$ for some simplex $\Delta$ containing the origin.
\begin{lemma}\label{lem:equality}
	Let $f\in\mathcal F(\mathbb R^n)$ be such that $\Vert f\Vert_\infty=f(0)$ and let $g:\R^n\to\R$ be
	$$
	g(x)=\int_0^\infty e^{-t}|K_t(f)\cap(x+ K_t(f))|dt.
	$$
	Then, for every $x\in\R^n$ and every $\lambda\in[0,1]$, $g(\lambda x)=g(0)^{1-\lambda} g(x)^\lambda$ if and only if $\frac{f(x)}{\Vert f\Vert_\infty}=e^{-\Vert x\Vert_{\Delta}}$ with $\Delta$ a simplex containing the origin.
\end{lemma}

\begin{proof}
	The condition $g(\lambda x)=g(0)^{1-\lambda} g(x)^\lambda$ for every $x\in\R^n$ and every $\lambda\in[0,1]$ implies that $g(x)\neq 0$ for every $x\in\R^n$, since $g$ is continuous at 0, as 0 is in the interior of the support of $g$. In the following, in order to ease the notation, we will denote $K_t=K_t(f)$ for every $t\in[0,\infty)$. For $x\in\R^n$ and $t\in[0,\infty)$, we define
	$$
	h_x(t)= e^{-t}|K_t\cap(x+K_t)|.
	$$
	Notice that for any $x\in\R^n$, any $t_1,t_2\in[0,\infty)$ and any $\lambda\in[0,1]$
	$$
	K_{(1-\lambda)t_1+\lambda t_2}\cap(x+K_{(1-\lambda)t_1+\lambda t_2})\supseteq(1-\lambda) (K_{t_1}\cap(x+K_{t_1}))+\lambda(K_{t_2}\cap(x+K_{t_2})).
	$$
	Therefore, by Brunn-Minkowski inequality
\begin{equation}\label{eq:Brunn-Minkowski}
	|K_{(1-\lambda)t_1+\lambda t_2}\cap(x+K_{(1-\lambda)t_1+\lambda t_2})|\geq|K_{t_1}\cap(x+K_{t_1})|^{1-\lambda}|K_{t_2}\cap(x+K_{t_2})|^\lambda
\end{equation}
	and
	\begin{eqnarray*}
		h_x((1-\lambda)t_1+\lambda t_2)&\geq &\left(e^{-t_1}|K_{t_1}\cap(x+K_{t_1})|\right)^{1-\lambda}\left(e^{-t_2}|K_{t_2}\cap(x+K_{t_2})|\right)^\lambda\cr
		&=&h_x(t_1)^{1-\lambda}h_x(t_2)^\lambda.
	\end{eqnarray*}
	Consequently, since $h_x$ is log-concave and integrable in $[0,\infty)$, for any $s\in[0,\infty)$, the set $\{t\in[0,\infty)\,:\,h_x(t)\geq s\}$ is either empty or a closed interval. Let us remark that for each $x\in\R^n$, since $g(x)\neq0$, the function $h_x(t)$ is not identically 0 and it attains its maximum at a unique point, since if $\Vert h_x\Vert_\infty=h_x(t_1)=h_x(t_2)$, with $t_1<t_2$, then we have that for every $\lambda\in[0,1]$ $h_x((1-\lambda)t_1+\lambda t_2)\geq\Vert h_x\Vert_\infty$, so $h_x((1-\lambda)t_1+\lambda t_2)=\Vert h_x\Vert_\infty$. Therefore, we have equality in \eqref{eq:Brunn-Minkowski} and by the characterization of the equality cases in Brunn-Minkowski inequality (see, for instance \cite[Section 1.2]{AGM}) $K_{t_2}\cap(x+K_{t_2})$ is a translation of $K_{t_1}\cap(x+K_{t_1})$ and they have the same volume. Thus $h_x(t_2)<h_x(t_1)$, which contradicts the fact that the maximum is attained both at $t_1$ and $t_2$. Notice also that for any $s\in[0,\infty)$ and any $\lambda\in[0,1]$, if $t_1\in \{t\in[0,\infty)\,:\,h_0(t)\geq s\Vert h_0\Vert_\infty\}$ and $t_2\in\{t\in[0,\infty)\,:\,h_x(t)\geq s\Vert h_x\Vert_\infty\}$, then calling $t=(1-\lambda)t_1+\lambda t_2$, since
	$$
	K_t\cap(\lambda x+ K_t)\supseteq(1-\lambda)K_{t_1}+\lambda (K_{t_2}\cap(x+K_{t_2})),
	$$
	by Brunn-Minkowski inequality
	$$
	|K_t\cap(\lambda x+ K_t)|\geq|K_{t_1}|^{1-\lambda}|K_{t_2}\cap(x+K_{t_2})|^\lambda
	$$
	and
	\begin{eqnarray*}
		h_{\lambda x}(t)&=&e^{-t}|K_t\cap(\lambda x+ K_t)|\geq(e^{-t_1}|K_{t_1}|)^{1-\lambda}(e^{-t_1}|K_{t_2}\cap(x+K_{t_2})|)^\lambda\cr
		&=&h_0(t_1)^{1-\lambda} h_x(t_2)^\lambda\geq s\Vert h_0\Vert_\infty^{1-\lambda}\Vert h_x\Vert_\infty^\lambda.
	\end{eqnarray*}
	Consequently, for any $x\in\R^n$ $s\in[0,\infty)$ and $\lambda\in[0,1]$.
	\begin{eqnarray*}
		&&\{t\in[0,\infty)\,:\,h_{\lambda x}(t)\geq s\Vert h_0\Vert_\infty^{1-\lambda}\Vert h_x\Vert_\infty^\lambda\}\supseteq\cr
		&&(1-\lambda) \{t\in[0,\infty)\,:\,h_0(t)\geq s\Vert h_0\Vert_\infty\}+\lambda \{t\in[0,\infty)\,:\,h_x(t)\geq s\Vert h_x\Vert_\infty\}.
	\end{eqnarray*}
	The last sets are non-empty for every $s\in[0,1]$. Thus, for any $x\in\R^n$ and $\lambda\in[0,1]$
	\begin{eqnarray*}
		\frac{g(\lambda x)}{\Vert h_0\Vert_\infty^{1-\lambda}\Vert h_x\Vert_\infty^\lambda}&=&\int_0^\infty|\{t\in[0,\infty)\,:\,h_{\lambda x}(t)\geq s\Vert h_0\Vert_\infty^{1-\lambda}\Vert h_x\Vert_\infty^\lambda\}|ds\cr
		&\geq&\int_0^1|\{t\in[0,\infty)\,:\,h_{\lambda x}(t)\geq s\Vert h_0\Vert_\infty^{1-\lambda}\Vert h_x\Vert_\infty^\lambda\}|ds\cr
		&\geq&(1-\lambda)\int_0^1|\{t\in[0,\infty)\,:\,h_0(t)\geq s\Vert h_0\Vert_\infty\}|ds\cr
		&+&\lambda\int_0^1|\{t\in[0,\infty)\,:\,h_x(t)\geq s\Vert h_x\Vert_\infty\}|ds\cr
		&=&(1-\lambda)\int_0^\infty\frac{h_0(t)}{\Vert h_0\Vert_\infty}dt+\lambda\int_0^\infty\frac{h_x(t)}{\Vert h_x\Vert_\infty}dt\cr
		&\geq&\left(\int_0^\infty\frac{h_0(t)}{\Vert h_0\Vert_\infty}dt\right)^{1-\lambda}\left(\int_0^\infty\frac{h_x(t)}{\Vert h_x\Vert_\infty}dt\right)^\lambda\cr
		&=&\left(\frac{g(0)}{\Vert h_0\Vert_\infty}\right)^{1-\lambda}\left(\frac{g(x)}{\Vert h_x\Vert_\infty}\right)^\lambda.\cr
	\end{eqnarray*}
	Since, by assumption $g(\lambda x)=g(0)^{1-\lambda} g(x)^\lambda$, all the inequalities in the last chain of inequalities are equalities and then
	\begin{itemize}
		\item $\Vert h_{\lambda x}\Vert_\infty=\Vert h_0\Vert_\infty^{1-\lambda}\Vert h_x\Vert_\infty^\lambda$,
		\item The following two sets are equal for very $s\in[0,1]$
		\begin{eqnarray*}
			&&\{t\in[0,\infty)\,:\,h_{\lambda x}(t)\geq s\Vert h_0\Vert_\infty^{1-\lambda}\Vert h_x\Vert_\infty^\lambda\}=\cr
			&&(1-\lambda) \{t\in[0,\infty)\,:\,h_0(t)\geq s\Vert h_0\Vert_\infty\}+\lambda \{t\in[0,\infty)\,:\,h_x(t)\geq s\Vert h_x\Vert_\infty\},
		\end{eqnarray*}
		\item $\int_0^\infty\frac{h_0(t)}{\Vert h_0\Vert_\infty}dt=\int_0^\infty\frac{h_x(t)}{\Vert h_x\Vert_\infty}dt$ or, equivalently
		$\frac{g(0)}{\Vert h_0\Vert_\infty}=\frac{g(x)}{\Vert h_x\Vert_\infty}$.
	\end{itemize}
	Notice that, since the sets in the second condition are intervals, if we call
	\begin{itemize}
		\item $t_1(s)=\min\{t\in[0,\infty)\,:\,h_0(t)\geq s\Vert h_0\Vert_\infty\}$
		\item $\overline{t}_1(s)=\sup\{t\in[0,\infty)\,:\,h_0(t)\geq s\Vert h_0\Vert_\infty\}$
		\item $t_2(s,x)=\min\{t\in[0,\infty)\,:\,h_x(t)\geq s\Vert h_x\Vert_\infty\}$
		\item $\overline{t}_2(s,x)=\sup\{t\in[0,\infty)\,:\,h_x(t)\geq s\Vert h_x\Vert_\infty\}$,
	\end{itemize}
	the second condition is equivalent to
	\begin{itemize}
		\item $h_{\lambda x}((1-\lambda)t_1(s)+\lambda t_2(s,x))=s\Vert h_0\Vert_\infty^{1-\lambda}\Vert h_x\Vert_\infty^\lambda$
		\item $h_{\lambda x}((1-\lambda)\overline{t}_1(s)+\lambda \overline{t}_2(s,x))=s\Vert h_0\Vert_\infty^{1-\lambda}\Vert h_x\Vert_\infty^\lambda$
	\end{itemize}
	and, since $h_0(t_1(s))=s\Vert h_0\Vert_\infty$, $h_x(t_2(s,x))= s\Vert h_x\Vert_\infty$, $h_0(\overline{t}_1(s))=s\Vert h_0\Vert_\infty$, and $h_x(\overline{t}_2(s,x))= s\Vert h_x\Vert_\infty$, the last two equalities imply that for any $s\in[0,1]$ there is equality in
	$$ |K_{(1-\lambda)t_1(s)+\lambda t_2(s,x)}\cap(\lambda x+ K_{(1-\lambda)t_1(s)+\lambda t_2(s,x)})|\geq|K_{t_1(s)}|^{1-\lambda}|K_{t_2(s,x)}\cap(x+K_{t_2(s,x)})|^\lambda$$
	and for any $s\in(0,1]$ there is equality in
	$$ |K_{(1-\lambda)\overline{t}_1(s)+\lambda \overline{t}_2(s,x)}\cap(\lambda x+ K_{(1-\lambda)\overline{t}_1(s)+\lambda \overline{t}_2(s)})|\geq|K_{\overline{t}_1(s)}|^{1-\lambda}|K_{\overline{t}_2(s,x)}\cap(x+K_{\overline{t}_2(s)})|^\lambda.$$
	
	This, implies that for any $x\in\R^n$, any $\lambda\in[0,1]$ and any $s\in(0,1]$
	\begin{itemize}
		\item $K_{t_2(s,x)}\cap(x+K_{t_2(s,x)})$ is a translation of $K_{t_1(s)}$ and also $K_{(1-\lambda)t_1(s)+\lambda t_2(s,x)}\cap(\lambda x+ K_{(1-\lambda)t_1(s)+\lambda t_2(s,x)})$ is a translation of $K_{t_1(s)}$,
		\item $K_{\overline{t}_2(s,x)}\cap(x+K_{\overline{t}_2(s,x)})$ is a translation of $K_{\overline{t}_1(s)}$ and also $K_{(1-\lambda)\overline{t}_1(s)+\lambda \overline{t}_2(s,x)}\cap(\lambda x+ K_{(1-\lambda)\overline{t}_1(s)+\lambda \overline{t}_2(s,x)})$ is a translation of $K_{\overline{t}_1(s)}$.
	\end{itemize}
	Let us remark that for $s=1$, since the function $h_x$ attains its maximum at a unique point and the value of this maximum is strictly positive, $t_1(1)=\overline{t_1}(1)$ and $t_2(1,x)=\overline{t}_2(1,x)$ for every $x\in\R^n$. Notice also that for any $x\in\R^n$ and any $\lambda\in[0,1]$, since $\Vert h_{\lambda x}\Vert_\infty=\Vert h_0\Vert_\infty^{1-\lambda}\Vert h_x\Vert_\infty^\lambda$ and $h_{\lambda x}((1-\lambda)t_1(1)+\lambda t_2(1,x))\geq \Vert h_0\Vert_\infty^{1-\lambda}\Vert h_x\Vert_\infty^\lambda$, we have that
	$$
	t_2(1,\lambda x)=(1-\lambda)t_1(1)+\lambda t_2(1,x).
	$$
	Since for every fixed $x\in\R^n$ and every $s\in(0,1]$ we have that $K_{t_2(s,x)}\cap(x+K_{t_2(s,x)})$ is a translation of $K_{t_1(s)}$,  $e^{-t_1(s)}|K_{t_1(s)}|=s\Vert h_0\Vert_\infty$, and $e^{-t_2(s,x)}|K_{t_2(s,x)}\cap(x+K_{t_2(s,x)})|=s\Vert h_x\Vert_\infty$, we obtain that for every $x\in\R^n$ and every $s\in[0,1]$
	$$
	\frac{s\Vert h_0\Vert_\infty}{e^{-t_1(s)}}=\frac{s\Vert h_x\Vert_\infty}{e^{-t_2(s,x)}}.
	$$
	Equivalently
	$$
	e^{-(t_2(s,x)-t_1(s))}=\frac{\Vert h_x\Vert_\infty}{\Vert h_0\Vert_\infty}.
	$$
	Thus, for every $x\in\R^n$, the difference $t_2(s,x)-t_1(s)$ does not depend on $s$ and $t_2(s,x)-t_1(s)=t_2(1,x)-t_1(1)$ for every $s\in[0,1]$.	
	Then, we have that for every $x\in\R^n$ and any $\lambda\in[0,1]$
	\begin{eqnarray*}
		t_2(0,\lambda x)&=&t_2(0,\lambda x)-t_1(0)=t_2(1,\lambda x)-t_1(1)=\lambda(t_2(1,x)-t_1(1))\cr
		&=&\lambda(t_2(0,x)-t_1(0))=\lambda t_2(0,x).
	\end{eqnarray*}
	Consequently, since for any $x\in\R^n$
	\begin{eqnarray*}
		t_2(0,x)&=&\inf\{t\in[0,\infty)\,:\,K_t\cap(x+K_t)\neq\emptyset\}\cr
		&=&\inf\{t\in[0,\infty)\,:\,x\in K_t-K_t\},
	\end{eqnarray*}
	we have that for any $x\in\R^n$ and any $\lambda\in[0,1]$
	$$
	\inf\{t\in[0,\infty)\,:\,\lambda x\in K_t-K_t\}=\lambda\inf\{t\in[0,\infty)\,:\,x\in K_t-K_t\}.
	$$
	Fixing $t_0\in[0,\infty)$ we have that for every $x\in\partial (K_{t_0}-K_{t_0})$ $\inf\{t\in[0,\infty)\,:\,x\in K_t-K_t\}=t_0$.
Otherwise, $\inf\{t\in[0,\infty):x\in K_t-K_t\}\leq t_1<t_0$. Since $\lambda x\notin K_{t_0}-K_{t_0}$ for every $\lambda> 1$,
we have that $\inf\{t\in[0,\infty):\lambda x\in K_t-K_t\}\geq t_0$, and thus
\[
\begin{split}
t_0 & \leq\inf\{t\in[0,\infty):\lambda x\in K_t-K_t\}\\
    & =\lambda\inf\{t\in[0,\infty):x\in K_t-K_t\}\\
    & \leq\lambda t_1,
\end{split}
\]
for every $\lambda>1$, a contradiction since $t_1<t_0$.
Thus, for any $\lambda\in[0,1]$, $\lambda x\in\partial(K_{\lambda t_0}-K_{\lambda t_0})$ and then for every $t\in[0,\infty)$ and any $\lambda\in[0,1]$
	$$
	\lambda K_t-\lambda K_t=\lambda(K_t-K_t)=K_{\lambda t}-K_{\lambda t}.
	$$
	Since $f$ is log-concave and $\Vert f\Vert_\infty=f(0)$ we have that $0\in K_t$ for every $t\in[0,\infty)$ and for every $t\in[0,\infty)$ and any $\lambda\in[0,1]$ $\lambda K_t\subseteq K_{\lambda t}$ and then
	$$
	\lambda K_t-\lambda K_t\subseteq K_{\lambda t}-K_{\lambda t}.
	$$
	Since the last inclusion is an equality and $\lambda K_t\subseteq K_{\lambda t}$ we obtain that $\lambda K_t=K_{\lambda t}$. Therefore, for every $t\in[0,\infty)$ $K_t=tK_1$ and $\frac{f(x)}{\Vert f\Vert_\infty}= e^{-\Vert x\Vert_{K_1}}$. Let us call $K:=K_1$. Since $t_1(0)=0$ and for every $x\in\R^n$
	\begin{eqnarray*}
		t_2(0,x)&=&\inf\{t\in[0,\infty)\,:\,K_t\cap(x+K_t)\neq\emptyset\}\cr
		&=&\inf\{t\in[0,\infty)\,:\,tK\cap(x+tK)\neq\emptyset\}\cr
		&=&\inf\{t\in[0,\infty)\,:\,x\in tK-tK\}\cr
		&=&\Vert x\Vert_{K-K},
	\end{eqnarray*}
	we have that for every $x\in\R^n$
	$$
	g(x)=g(0)\frac{\Vert h_x\Vert_\infty}{\Vert h_0\Vert_\infty}=g(0)e^{-(t_2(0,x)-t_1(0))}=g(0)e^{-\Vert x\Vert_{K-K}}.
	$$
	Therefore,
	$$
	\int_{\R^n}g(x)=(n!)^2|K||K-K|.
	$$
	On the other hand, if $\frac{f(x)}{\Vert f\Vert_\infty}=e^{-\Vert x\Vert_K}$
	\begin{eqnarray*}
		\int_{\R^n}g(x)dx&=&\int_{\R^n}\int_{\R^n} \min\{e^{-\Vert x\Vert_K},e^{-\Vert y\Vert_K}\}dydx\cr
		&=&\int_{\R^{2n}}e^{-\Vert (x,y)\Vert_{K\times K}}dydx\cr
		&=&(2n)!|K|^2.
	\end{eqnarray*}
	Thus
	$$
	{\binom{2n}{n}}|K|=|K-K|
	$$
	and, since we have equality in Rogers-Shephard inequality \cite{RS}, $K$ is a simplex and $0\in K=K_1$.
\end{proof}

Lemma \ref{lem:g(tx)=g(x)^t} and Lemma \ref{lem:equality} together characterize the equality case.
%\begin{proof}[Proof of equality conditions in Theorem \ref{thm:FunctionalZhang}]
%Since $f$ attains equality in Theorem \ref{thm:FunctionalZhang}, if
%\[
%g(x)=\int_0^\infty e^{-t}|K_t(f)\cap(x+ K_t(f))|dt,
%\]
%then Lemma \ref{lem:g(tx)=g(x)^t} yields that for every $x\in\R^n$ and every $\lambda\geq0$, then $g(\lambda x)=g(x)^\lambda g(0)^{1-\lambda}$.
%By Lemma \ref{lem:equality}, this occurs if and only if $\frac{f(x)}{||f||_\infty}=e^{-||x||_\Delta}$, for some simplex $\Delta$ containing the origin.
%\end{proof}

%\begin{rmk}
%	An analogous although independent extension to Theorem \ref{thm:FunctionalZhang} should be
%	$$
%	\Vert f\Vert_\infty\Vert f\Vert_1^{n-1}|\Pi^*(f)|\geq\frac{\binom{2n}{n}}{2^nn^n},
%	$$
%	with equality if and only if $f$ is the characteristic function of a simplex.
%\end{rmk}

\end{document}